\documentclass{article}
\usepackage{algorithm}
\usepackage{amsthm}
\usepackage{amsmath}
\usepackage{amsfonts}
\usepackage{algorithmic}
\usepackage{geometry}
\usepackage{color}
\usepackage{graphicx}
\newtheorem{thm}{Theorem}

\newtheorem{remark}{Remark}
\newtheorem{lem}{Lemma}

\usepackage{color}
\usepackage{hyperref}

\title{Convex Shape  Priors for Level Set Representation}
\author{Shousheng Luo and Xue-Cheng Tai
\thanks{Shousheng Luo is with  Beijing Computational Science Research Center, Haidian, Beijing, 100193, and also
School of Mathematics and Statistics, Henan University, 475001, Kaifeng, China, and
Department of Mathematics, Hong Kong University of Science and Technology, Kowloon, Hong Kong.
e-mail: sluo@henu.edu.cn. }
\thanks{Xuecheng Tai is with Department of Mathematics,  Hong Kong Baptist University, Kowloon Tong, Hong Kong, e-mail: tai@mi.uib.no.}
}
\begin{document}
\maketitle
\begin{abstract}
For many applications, we need to use techniques to represent convex shapes and objects. In this work, we use level set method to represent shapes and find a necessary and sufficient condition on the level set function to guarantee the convexity of the represented shapes. We take image segmentation as an example to apply our technique. Numerical algorithm is developed to solve the variational model.
In order to improve the performance of segmentation for complex images, we also incorporate landmarks into the model. One option is to specify points that the object boundary must contain. Another option is to specify points that the foreground (the object) and the background must contain.
Numerical experiments on different
images validate the efficiency of the proposed models and algorithms. We want to emphasize that the proposed technique could be used for general shape optimization with convex shape prior.
For other applications, the numerical algorithms need to be extended and modified.
\end{abstract}


\section{Introduction}
Image segmentation is a fundamental task in image sciences and arises from a wide range of applications, such as computer vision, medial imaging and analysis.
Numerous models and algorithms have been proposed for this problem \cite{Mumford1989},\cite{Chan2001},\cite{Caselles1995},\cite{Ashburner2005Unified}.
 These models can be roughly categorized into region based methods and edge based methods. Active contour method proposed in \cite{Caselles1995}
is one of the famous edge based models. Mumford-Shah's (MS) functional
is a fundamental region based method \cite{Mumford1989}.
Different variants and approximation models were proposed.
The Chan-Vese (CV) model, proposed in \cite{Chan2001}, is exactly a piece-wise constant case of the MS model.

Existing models are usually based on the image intensity values.
These approaches can not present desirable segmentation results for complex real images because the interested objects in the images are usually occluded by others, or can not be distinguished from others in low contrast images. For example in medical imaging, the organs, such as liver, kidney and heart, usually have similar intensity values in computed tomography image.
This turns out that a prior shape of the object should be incorporated in a
proper way to get a meaningful segmentation especially under low contrast, occlusions and noisy conditions.

Various image segmentation models with shape priors were proposed in the literature. In \cite{Leventon2003Statistical}, the
geodesic active contour model is extended by incorporating shape information into the contour evolution process.
Because an object in the plane corresponds a unique signed distance function (SDF, a special level set function),  shape priors
are usually described by the corresponding SDF.

In \cite{Rousson2002Shape},
an energetic model is proposed to incorporate shape prior by level set representation, which can deal with noisy, occluded and corrupted image segmentation problem. In \cite{Cremers2003Towards}, a labelling function is introduced to
enforce shape prior. This method is developed in \cite{Chan2005Level} to allow scaling, translation and rotation of prior shape.
In \cite{Thiruvenkadam2007Segmentation}, a
model incorporating shape prior knowledge is proposed
for multiple objects segmentation under occlusions and subtle boundary condition.
These approaches mentioned above are all based on a shape training set.
The studies on general shape priors, such as star shape and  convex shape, attracted more and more attentions recently \cite{Gorelick2017Convexity},\cite{Yan2018},\cite{Bae2017Augmented},
\cite{Royer2015Convexity},\cite{Yuan2012},\cite{Yuan2012An}.

In this paper, we focus on the models and algorithms for image
segmentation with convexity prior.
The importance of convexity for shape completion has attracted attentions for a long time, see
\cite{Liu1999The}. In practice, a lot of objects are convex, such as
buildings, organs and cells. Several approaches have been proposed for
image segmentation with convexity prior.
In \cite{Royer2015Convexity}, a graph-based model is proposed for multiple phases
segmentation with convexity constraints by extending the minimum cost multi-cut method.

In \cite{Gorelick2017Convexity}, a
discrete segmentation model is proposed by incorporating convexity prior.
Based on the definition of convexity,  i.e. the
line segment between any two points belonging to the shape shall be inside it, a  convexity model is proposed in \cite{Gorelick2017Convexity} by penalizing 1-0-1 configuration on all line segments in the image domain. Here 1 or 0 denotes the pixel belonging to the object or not.
Efficient algorithm based on trust region approach is investigated for the proposed model by using linear and quadratic approximations.

In \cite{Bae2017Augmented},
an $L^1$-Euler Elastica energy-based model is studied.
By imposing a large weight on the Euler Elastica regularization term,
 one can get convex shape segmentation result
 because the absolute curvature integral along any
 closed curve is larger than $2\pi$ unless the curve is convex.

In \cite{Yan2018}, it is  proved that the shape convexity is
guaranteed by the convexity of its corresponding signed distance function, which is equivalent
to the nonnegativity of its Laplacian (a linear inequality constraint).
This representation technique of convexity shape is incorporated with Chan-Vese model \cite{Chan2001} for
two-phase image segmentation. Therefore, the proposed model in \cite{Yan2018}
contains two constraints, both of which make the solution to be
a convex signed distance function.
In \cite{Yan2018}, a numerical method is used to handle the Chan-Vese functional minimization and
the two constraints separately.


In this work, we develop and improve the method in \cite{Yan2018} theoretically and numerically. 
We prove the equivalence between the shape convexity and the Laplacian nonnegativity of its corresponding SDF by geometric method.
For a given shape, we prove that all the sublevel sets of its SDF are convex if and only if the shape is convex. Therefore, all level sets of the SDF are
convex curves, and their curvatures, which are computed by the Laplacian of SDF, are nonnegative.

The proposed equivalence between convex shape and its SDF could be used for general shape optimization problems. In this paper, we take image segmentation problem as an example to apply this technique.
We consider probability-based models for two-phase image segmentation with convexity shape prior.
The probabilities are estimated by Gaussian mixture method (GMM) and the similarities between the given priors
of foreground and background.
Moreover,
in order to improve segmentation results for challenging images we
incorporate boundary landmarks with Gaussian mixture method (GMM).

{ For the sake of computational efficiency, the convex constraint is imposed on a subregion (containing the segmented object) of the image domain, and proper boundary conditions are given on the image boundary.}
By introducing two auxiliary variables, alternate direction method of multiplier (ADMM),
which is widely used in image sciences \cite{Wu2010},\cite{Tai2011A},\cite{Zhu2013Image},\cite{Bae2017Augmented},
is used to solve the proposed models. The proposed algorithm can handle the objective energy functional and
constraints simultaneously, and all variables have closed form solutions except the SDF update in the iteration procedure. SDF update needs to solve a fourth order  partial differential equation with the given condition,
which can be converted to two second order partial differential equations and solved by
discrete cosine transform (DCT) efficiently \cite{Strange1999DCT}.


The rest of this paper is organized as follows.
Theoretical analysis about convex shape representation by SDF is presented in Section \ref{sec2}.
Probability-based models with convexity prior and  two probability estimation methods are given in  Section \ref{sec3}.
Section \ref{sec4} is devoted to numerical algorithms for the proposed models. Numerical results for various images are illustrated in Section \ref{sec5}.
Lastly, some conclusions and future work are discussed in Section \ref{sec6}.

\section{Convex shape prior with signed distance functions}\label{sec2}
Level set representation method \cite{Osher1988Fronts} is  one of the most
popular tools in the field of image segmentation  \cite{Chan2001},\cite{Chan2005Level},\cite{Tai2017simple}
because it can handle the topological changes of curves and surfaces efficiently.
It is well-known that any object in the plane corresponds
a unique signed distance function (SDF) --
a special level set function \cite{Osher1988Fronts}, and vice versa.
Therefore, the properties of the object can be described by its
signed distance function. In \cite{Yan2018},
it is proved that the convexity of an object is guaranteed by the convexity of its corresponding SDF, which is equivalent to nonnegativity of the Laplacian of SDF.
We will prove a further conclusion by geometric method in this section.

The signed distance function of an object $\Omega_0$ is defined as
\begin{eqnarray}
\phi(x)=\left\{\begin{array}{ll}
-dist(x,C)&x~~\text{inside}~~ C\\
dist(x,C)&x~~\text{outside}~~C
\end{array}
\right.,
\end{eqnarray}
where $C=\partial\Omega_0$ is the boundary of $\Omega_0$
and $dist(x,C)=\min_{y\in C}\|x-y\|_2$.
It is well known that
\begin{eqnarray}
&&|\nabla\phi|=1,
\end{eqnarray}
holds almost everywhere for the SDF of any object (or curve).

The level set and sublevel set of a given function $f$ are defined as
\begin{eqnarray}
lev_f^c&=&\{x|f(x)=c\},\\
slev_f^c&=&\{x|f(x)\leq c\}.
\end{eqnarray}
It is obvious that $C=lev_\phi^0$, and $\Omega_0=slev_{\phi}^0$. The following theorem was proven in \cite{Yan2018}.
\begin{thm}\label{ThYan}
Let $\phi$ be the signed distance function of an object $\Omega_0\subset \Omega$. If $\phi\in C^2$ a.e. in $\Omega$ and satisfies the following condition:
\begin{eqnarray}\label{eq:tai1}
\triangle \phi\geq0,~~ a.e.~~ x\in \Omega.
\end{eqnarray}
Then object $\Omega_0$ must be convex.
\end{thm}
In Theorem \ref{LuoTh}, we will prove that the convexity of shape is equivalent to
 Laplacian nonnegativity of its corresponding SDF, i.e. condition (\ref{eq:tai1}) is sufficient and also necessary for the convexity of the object $\Omega_0$.
 We prove the following Lemma  first. %

\begin{lem}\label{LuoLM}
Let $\phi$ be the signed distance function of $\Omega_0 \subset\mathbb{R}^2$,
and $\Omega_c$ is the sublevel set of $\phi$ with a given number $c$.
Then we have that $\Omega_0$ is convex if and only if $\Omega_c$ is convex for all $c$.
\end{lem}
\begin{proof}
We prove this conclusion for nonempty $slev_\phi^c$ only.
It is obvious that $\Omega_0$ is convex if $\Omega_c$ are convex for all $c$.
We will prove that $\Omega_c$ are convex if $\Omega_0$ is convex in the following.
\\
(\romannumeral1)~For $c=0$, it is obvious.\\
(\romannumeral2) For $c>0$, suppose any two points $x_1,x_2\in slev_{\phi}^c$, i.e.
$dist(x_j)\leq c~(j=1,2)$. We will prove that $\lambda{x_1}+(1-\lambda)x_2 \in slev_{\phi}^c$ for $0\leq\lambda\leq1$, i.e.
\begin{eqnarray}
\phi(\lambda{x_1}+(1-\lambda)x_2)\leq c.
\end{eqnarray}
There are three cases to be considered: A) $x_1,x_2\in\Omega_0$; B) only one point belonging to $\Omega_0$; C) $x_1,x_2$ outside $\Omega_0$.
\\
A) It is obvious that  $\lambda{x}_1+(1-\lambda)x_2\in\Omega_0$ for $\lambda\in[0,1]$ by the assumption that $\Omega_0$ is convex. Therefore, we have
\begin{eqnarray}
\phi(\lambda{x_1}+(1-\lambda)x_2)\leq0<c.
\end{eqnarray}
B) Without loss of generality, we assume $x_1\in\Omega_0$ and $x_2\not\in\Omega_0$.
Suppose $\hat{\lambda}\in [0,1]$ such that $\hat{x}=\hat{\lambda}x_1+(1-\hat{\lambda})x_2\in C=\partial\Omega_0$.
For $\lambda\in[\hat{\lambda},1]$, we have $\lambda{x_1}+(1-\lambda)x_2\in \Omega_0$, and
\begin{eqnarray}
\phi(\lambda{x_1}+(1-\lambda)x_2)\leq0<c.
\end{eqnarray}
For $\tilde{\lambda}\in[0,\hat{\lambda})$, we have $\tilde{x}=\tilde{\lambda}{x_1}+(1-\tilde{\lambda})x_2=\hat{x}+(\hat{\lambda}-\tilde{\lambda})(x_2-x_1)$ outside $\Omega_0$.
Because $x_1,x_2$ and $\hat{x}$ are collinear, there is $0\leq\mu\in(0,1]$ such that
$\tilde{x}=\hat{x}+\mu(x_2-\hat{x})$, and $\tilde{x}=x_2$ when $\mu=1$.
Suppose $y_2\in\partial\Omega_0$ such that
\begin{eqnarray}
\phi(x_2)=\|x_2-y_2\|_2.
\end{eqnarray}
By the assumption, we have $\tilde{y}=\mu y_2+(1-\mu)\hat{x}\in slev_\phi^0$.
It is obvious that for any $x$ outside $\Omega_0$ we have
\begin{eqnarray}
\phi(x)=\min_{y\in C}\|x-y\|_2
=\min_{y\in slev_\phi^0}\|x-y\|_2. \label{eq:dist1}
\end{eqnarray}
Therefore, we have
\begin{eqnarray}
\phi(\tilde{x})&=&\min_{y\in slev_\phi^0}\|\tilde{x}-y\|_2\notag\\
&\leq&\|\tilde{x}-\tilde{y}\|_2\nonumber\\
&=&\|\mu(x_2-y_2)\|_2\notag\\
&\leq&\|x_2-y_2\|_2\nonumber\\
&=&\phi(x_2)=c.
\end{eqnarray}
The second inequality held by $\mu\in(0,1]$.
\\
C)
Let $y_{j}~(j=1,2)$ be the points on $C=\partial\Omega_0$ such that $\phi(x_j)=\|x_j-y_j\|_2~(j=1,2)$.
Because $slev_\phi^0$ is convex, $\lambda{y_1}+(1-\lambda)y_2\in slev_\phi^0$ for $\forall~ \lambda\in[0,1]$.
By (\ref{eq:dist1}), we have
\begin{eqnarray}
&&\phi(\lambda{x_1}+(1-\lambda)x_2)\nonumber\\
&=&\min_{y\in slev_\phi^0 }\|\lambda{x_1}+(1-\lambda)x_2-y\|_2\notag\\
&\leq& \|\lambda{x_1}+(1-\lambda)x_2-\lambda{y_1}-(1-\lambda)y_2\|_2\notag\\
&\leq&\lambda\|x_1-y_1\|_2+(1-\lambda)\|x_2-y_2\|_2\notag\\
&\leq&c.
\end{eqnarray}
(\romannumeral3) For $c<0$, suppose $x_j~(j=1,2)$ in $slev_{\phi}^c$, and $y_j\in C~(j=1,2)$ such that
\begin{eqnarray}
\phi(x_j)&=&-\min_{y\in C}\|x_j-y\|_2\nonumber\\
&=&-\|x_j-y_j\|_2\leq c,~j=1,2.
\end{eqnarray}
It is obvious that $\|x_j-y_j\|_2\geq -c~(j=1,2)$.
We have $\lambda{x}_1+(1-\lambda)x_2\in slev_{\phi}^c$ for all $\lambda\in[0,1]$, i.e.
\begin{eqnarray}
\phi(\lambda{x_1}+(1-\lambda)x_2)&=&-\min_{y\in{C}}
\|\lambda{x_1}+(1-\lambda)x_2-y\|_2\nonumber\\
&\leq& c,~~\forall~ \lambda\in [0,1].
\end{eqnarray}
In fact, if there is a number $\hat{\lambda}\in [0,1]$ such that $\hat{x}=\hat{\lambda}{x_1}+(1-\hat{\lambda})x_2\not\in slev_\phi^c$, i.e.  $\phi(\hat{x})=\phi(\hat{\lambda}{x_1}+(1-\hat{\lambda})x_2)>c$, there is a point $\hat{y}$ on $C$ such that
\begin{eqnarray}
d=\|\hat{\lambda}{x_1}+(1-\hat{\lambda})x_2-\hat{y}\|_2<-c.
\end{eqnarray}
Let us move the line segment $\{\lambda{x_1}+(1-\lambda)x_2|0\leq\lambda\leq1\}$ $d$ parallel  such that $\mathcal{M}(\hat{\lambda}{x_1}+(1-\hat{\lambda})x_2)=\hat{y}$, where $\mathcal{M}$ denotes the moving operator
\begin{eqnarray}
\mathcal{M}(x)=x+(\hat{y}-\hat{x}).
\end{eqnarray}
Therefore, for any two points $z_1,z_2$ and $\alpha\in\mathbb{R}$,  we have
\begin{eqnarray}
\mathcal{M}(\alpha z_1 +(1-\alpha) z_2)=\alpha \mathcal{M}(z_1)+(1-\alpha)\mathcal{M}(z_2).
\end{eqnarray}
Because the distances between $x_j~(j=1,2)$ and $C$
are $-\phi(x_j)\geq-c>d$, $\mathcal{M}(x_j)~(j=1,2)$ are inside $C$ still,
i.e. $\mathcal{M}(x_i)~(j=1,2)$ are interior points of  $slev_\phi^0$.
Therefore,  for all $\lambda\in[0,1]$ we have    $\mathcal{M}(\lambda{x_1}+(1-\lambda)x_2)=\lambda\mathcal{M}(x_1)+(1-\lambda)\mathcal{M}(x_2)$ are interior points of $slev_\phi^0$ due to the convexity of
$slev_\phi^0$, which contradicts to $\mathcal{M}(\hat{\lambda}{x_1}+(1-\hat{\lambda})x_2)=\hat{y}$ on $C$.
\end{proof}
By the conclusion of Lemma \ref{LuoLM}, all the level set curves of the SDF of a convex object are convex.
It is well known that
the convexity of curve is equivalent to the nonnegativity of its curvature.
Therefore, we have the following theorem.
\begin{thm}\label{LuoTh}
Let $\phi$ be a signed distance function of $\Omega_0$.
If $\phi\in C^2$ almost everywhere (a.e.), we have the convexity of $\Omega_0$ is equivalent to $\triangle \phi\geq0$ almost everywhere.
\end{thm}
\begin{proof}
It is known that
the curvature of the level set curve of SDF $\phi$  is
\begin{eqnarray}
\kappa=\nabla\cdot\left({\nabla\phi\over|\nabla\phi|}\right)=\triangle\phi.
\end{eqnarray}
The second equality held by the fact that $|\nabla\phi|=1$ holds almost everywhere for
signed distance function.
Using the fact that  the curvature of a curve is nonnegative if and only if
it is convex, we get the conclusion.
\end{proof}


\section{Image segmentation models with convexity shape constraint}\label{sec3}
In this section, we  present some probability-based models for two-phase image segmentation with convex shape prior.
We will  present the general probability-based model first.
{Generally speaking, the probabilities showing a point belong to foreground/background
will be  estimated by Gaussian mixture method. }
In case that we have some points or subregions of foreground and background are already labeled,
we will calculate the probabilities for each point using similarities between this point and
the labeled points.
In order to improve the performance of the GMM-based model,
we will also enhance the model by adding landmarks on the object boundary.
Numerical algorithms for the proposed models will be presented in Section \ref{sec4}.

Let $I: x\in\Omega \mapsto I(x)\in\mathbb{R}^d$~($d=1$ for grey image and $d=3$ for color image) be an image defined on a square domain $\Omega$.
Denote the object region by $\Omega_0$ and this is the region we need to identify.
A general two-phase image segmentation model with convexity shape prior can be written as
\begin{multline}
\arg\min\limits_{\Omega_0}\int_{\Omega_0} f_{I,0}(x)dx+\int_{\Omega\setminus \Omega_0}f_{I,1}(x)dx+\int_{\partial\Omega_0}g(C(s))ds, \\
\text{subject to}~~\Omega_0~~\text{ being convex},
\label{moddom}
\end{multline}
where $f_{I,i}(x)~(i=0,1)$ are  the probability-related
functions and $C(s)=\partial\Omega_0$ is the object boundary, $g$ is usually an edge detection function.
This model is often called the 'Potts model' and we are adding the convex shape prior here.

As was done in \cite{Lie2006}\cite{Glowinski2016Some}, we  use binary label function
to represent the object.  Let $u(x)=0$ for $x\in\Omega_0$ and $1$ otherwise. The model
(\ref{moddom}) can be rewritten as
\begin{eqnarray}
\begin{array}{l}
\arg\min_{u\in\{0,1\}}\int_\Omega f_I(x)u(x)dx
+\int_\Omega g(x)|\nabla u|dx\\
\text{subject to }~~ \Omega_0 = \{x|u(x)=0\}~~\text{being convex},
\end{array} \label{modu}
\end{eqnarray}
where
$f_I(x)=f_{I,1}(x)-f_{I,0}(x)$ is called region force term. The edge detector $g$ is computed by
\begin{eqnarray}
g(x)={\alpha\over1+\beta|\nabla G\ast{I}(x)|},\label{eqG}
\end{eqnarray}
where $\alpha,\beta>0$ are two parameters, and $G$ is a smoothing function (Gaussian kernel function for example) to
suppress the noisy effect.

The region force term  plays an important role in the performance of image segmentation model (\ref{modu}).
Probability based method is one of the widely investigated approaches in the literature for image segmentation and data clustering due to its flexibility and
robustness of intensities \cite{Ashburner2005Unified},\cite{Yin2017Effective},\cite{Wei2017New},
\cite{Khadeeja2018Auto},\cite{Xiong2016},\cite{Marco2008}.

For a given image $I(x)$, suppose the probabilities of point $x$ belonging to the
two phases are $p_{I,i}(x)(i=0,1)$, respectively.
It is obvious $p_{I,0}(x)+p_{I,1}(x)=1$.
Then binary indicator function  $u(x)$  for the object $\Omega_0$ is a random variable:
\begin{eqnarray}
p\{u(x)|I(x)\}&=&\left\{\begin{array}{ll}
p_{I,0}(x)&u(x)=0\\
p_{I,1}(x)&u(x)=1
\end{array}\right.\notag\\
&=&[p_{I,1}(x)]^{u(x)}[1-p_{I,1}(x)]^{1-u(x)},
\end{eqnarray}
where $p_{I,0}(x)=1-p_{I,1}(x)$.
Assuming $u(x)$ at all $x\in\Omega$ are independent, we have
\begin{eqnarray}
p\{u|I\}=\prod_{x\in\Omega}p\{u(x)|I(x)\}.
\end{eqnarray}
The region force term is usually computed by the negative likelihood function of $p\{u(x)|I(x)\}$,
\begin{eqnarray}
-\ln(p\{u(x)|I(x)\})=u(x)[-\ln(p_{I,1}(x))+\ln(1-p_{I,1}(x))],\nonumber
\end{eqnarray}
i.e. $f_I(x)=-\ln(p_{I,1}(x))+\ln(1-p_{I,1}(x))$.
In practice, in order to fit the boundary precisely, we modify the region force term as weighted summation as
\begin{eqnarray}
f_I(x)=-w_1\ln(p_{I,1}(x))+w_0\ln(1-p_{I,1}(x))].\label{eqGRF}
\end{eqnarray}

Suppose $\phi$ is the signed distance function of $\Omega_0$. We have
\begin{eqnarray}
u(x)=H(\phi(x))=\left\{\begin{array}{ll}
1&\phi(x)>0\\
0&\phi(x)\leq0
\end{array}
\right.,
\end{eqnarray}
where $H(s)$ is the Heaviside function, i.e. $H(s)=1$ for $s>0$ and $H(s)=0$ for $s\leq0$.
By Theorem \ref{LuoTh} in Section \ref{sec2} and the following equality
\[|\nabla u| = \delta (\phi) |\nabla \phi | = \delta(\phi) ,\]
where $\delta(\phi)=H^\prime(\phi)$ is Dirac distribution function,
the model (\ref{modu}) can be
re-written as
\begin{eqnarray}
\arg\min_\phi\int_\Omega F(\phi)dx, ~~\triangle\phi\geq0,|\nabla\phi|=1,\label{modconv}
\end{eqnarray}
where $
F(\phi)=g(x)\delta(\phi)+f_I(x)H(\phi)$.
Both constraints in (\ref{modconv}) make the solution to be convex
signed distance function.
The constraint $|\nabla\phi|=1$ is to guarantee that $\phi$
is a signed distance function, and the constraint $\triangle\phi\geq0$
makes $\phi$  to be convex by Theorem \ref{LuoTh}.

There are a lot of methods to estimate the probabilities for the model above.
In order to highlight the proposed method for convex shape representation and numerical algorithm,
we use two simple methods do estimate the probabilities.
\subsection{Gaussian mixture method}
In this subsection, we will present the Gaussian mixture method (GMM) to
estimate the probabilities $p_{I,i}(i=0,1)$ mentioned above for a given image $I$.
Assume $I(x)$ obeys
mixed  Gaussian distributions
$G(\mu_0,\Sigma_0)$ and $G(\mu_1,\Sigma_1)$, i.e.
\begin{eqnarray}
p(I(x))=c_0p_0(I(x))+c_1p_1(I(x)),
\end{eqnarray}
where $c_i~(i=0,1)$ are the proportions $(c_0+c_1=1)$ of
two distributions, and
\[p_i(I(x))={1\over(2\pi)^{d\over2}\det(\Sigma_i)^{1\over2}}
e^{-{1\over2}\|I(x)-\mu_i)\|_{\Sigma_i^{-1}}^2
},i=0,1.\]
We have
\[
p\{u(x)\}=\left\{\begin{array}{ll}
     c_0& u(x)=0 \\
     c_1& u(x)=1
\end{array}
\right.
,\]
and  $c_1={{1\over|\Omega|}\int_{\Omega}u(x)dx},~c_0=1-c_1$, where $|\Omega|=\int_{\Omega}1dx$.
By Bayesian method, we have the probability
of  point $x$ belonging to the $1$st class $u(x)=1$ is
\begin{eqnarray}
p_{I,1}(x)&=&\{u(x)=1|I(x)\}\nonumber\\
&=&{p\{I(x)|u(x)=1\}p\{u(x)=1\}\over p(I(x))}\nonumber\\
&=&{c_1p_1(I(x))\over{c_0p_0(I(x))+c_1p_1(I(x))}}, \label{GMM}
\end{eqnarray}
and $p_{I,0}(x)=1-p_{I,1}(x)$ for zeroth class $u(x)=0$. Therefore, we can
obtain the region force term by (\ref{eqGRF})
\begin{eqnarray}\label{eq:tai2}
f_{I}(x)=-w_1\ln{p_{I,1}(x)}+w_0\ln(1-p_{I,1}(x)).\label{eq:RegGMM}
\end{eqnarray}
Therefore, the GMM-based model is to minimize the following energy functional:
\begin{eqnarray}
\arg\min_{\phi,c_i,\mu_i,\Sigma_i,i=0,1}
\int_{\Omega}F(\phi)dx,|\nabla\phi|=1,\triangle\phi\geq0,\label{modGMM}
\end{eqnarray}
where $F(\phi)=H(\phi)f_{I}+g(x)\delta(\phi)$ and $f_{I}$ computed by (\ref{eq:RegGMM}).

We can get segmentation result by iteration method.
Suppose we have a
segmentation result $\phi$ (initial one is given by user). Similar to expectation maximization (EM) method \cite{McLachlan2007EM}, we can
estimate the probabilities $p_{I,i}~(i=0,1)$.
The binary function  $H(\phi(x))$ can be viewed as the probability of the point $x$ belonging to the $1$st class.
Firstly, we can estimate the parameters
$c_i,\mu_i$ and $\Sigma_i~(i=0,1)$:
\begin{eqnarray}
c_i&=&{1\over|\Omega|}\int_\Omega{q_{i}(x)}dx,\label{wUp}\\
\mu_i&=&{\int_\Omega{q_{i}(x)I(x)}dx\over\int_\Omega q_{i}(x)dx},
\label{muUp}\\
\Sigma_i&=&{\int_\Omega{q_{i}(x)(I(x)-\mu_i)^T(I(x)-\mu_i)}dx
\over\int_\Omega q_{i}(x)dx}\label{sigmaUp},
\end{eqnarray}
where $q_{1}(x)=H(\phi),q_{0}(x)=1-q_{1}(x)$.
After having the parameters above, we can obtain two Gaussian
distributions $G(\mu_i,\Sigma_i)~(i=0,1)$, and compute the posterior
probabilities $p_{I,i}=p\{u=i|I(x)\}~(i=0,1)$ by (\ref{GMM})
and region force term for image segmentation by (\ref{eq:RegGMM}).
A new level set function $\phi$ can be obtained by minimizing
 (\ref{modGMM}) with region force term fixed.
\subsection{Gaussian mixture method with boundary landmarks}
In order to improve the performance of GMM-based model, we
incorporate object boundary priors into the model (\ref{modGMM}).
For nonuniform, complex and low contrast images, it is difficult to
extract the object boundary robustly and precisely.
 Although it is hard and time consuming to draw
 the whole object boundary manually, it is often very easy to
use  prior knowledge to determine some points on the object boundary. In the following, we will present a model with object boundary landmarks and convex shape prior.

Let us assume that the object boundary must pass through
points $x_k~(k=1,2,\cdots,K)$. With the level set representation, this is true if and only
if  $\phi(x_k)=0,\ k=1,2\cdots K$. Correspondingly, the GMM model (\ref{modGMM}) with boundary landmarks is to solve the following constrained minimization problem:
\begin{eqnarray}
\begin{array}{l}
\arg\min\limits_{\phi,c_i,\mu_i,\Sigma_i,i=0,1}\int_{\Omega}F(\phi)dx, \triangle\phi\geq0, |\nabla\phi|=1, \\
~~~~~~~~~~~~~~~~~~~~~~~~~~~~\phi(x_k)=0,  k=1,2,\cdots,K.
\end{array}
\end{eqnarray}
We use the penalization method to handle the constraints $\phi(x_k) =0,k=1,2,\cdots,K$ and relax the above minimization problem as
\begin{eqnarray}
\arg\!\min_{{\phi,c_i,\mu_i,\Sigma_i,i=0,1}}\!\int_{\Omega}F_{L}(\phi)dx, \triangle\phi\geq0, |\nabla\phi|=1,\label{mod_LE}
\end{eqnarray}
where $
F_{L}(\phi)=F(\phi)+{\theta\over2}\sum_{k=1}^K |\phi(x_k)|^2
$
with $\theta$ being the penalization parameter which is normally taken as a fixed large positive number.
\subsection{Prior region-based method}
It is becoming an usual method to label subregions of foreground and background to
alleviate the difficulties for complex images segmentation .
In this subsection, we use a method based on the labelled
priors to estimate the probabilities, which are not updated in the implementation procedure.

Assume $R_{bg}$ and $R_{ob}$ are the given region priors for background and
foreground (object) of the image, i.e. $R_{bg}$ and $R_{ob}$ are
parts of background and foreground, respectively.
For any point in the image domain,
we compute the probabilities belonging to background and foreground for it by the similarities
between it and all the labelled points of  $R_{bg}$ and $R_{ob}$.
The similarity between any two points  $x,y$ is computed by their locations and
image intensity values
\begin{eqnarray}
S_I^a(x,y)=e^{-a_1\|x-y\|_2^2-a_2\|I(x)-I(y))\|_2^2},
\end{eqnarray}
where $a_1,a_2$ are two numbers to balance their location and intensity distances.
In this paper, we adopt $a_1=0.1$, $a_2=10$.
The probability belonging to background $(i=1)$ is defined as
\begin{eqnarray}
p_{I,1}(x)=p\{u(x)=1\}=
{\int_{R_{bg}}S_I^a(x,y)dy\over\int_{R_{bg}\bigcup R_{ob}}S_I^a(x,z)dz},\label{eqLprb}
\end{eqnarray}
and $p_{I,0}(x)=1-p_{I,1}(x)$ is the probability  belonging to foreground (object).
If the denominator in (\ref{eqLprb}) is less than a
threshold $\epsilon_p$ (0.01 in this paper),
we set $p_{I,0}(x)=p_{I,1}(x)=0.5$.
Furthermore, for the probabilities of the labelled points, we set
$p_{I,1}(x)=1 $ for $x\in R_{bg}$ and $p_{I,1}(x)=0$ for $x\in R_{ob}$,
and vice versa.
After computing the probabilities $p_{I,i}~(i=0,1)$,
the region force term for the image segmentation model (\ref{modconv}) is then taken as
\begin{eqnarray}
f_I(x)=-w_1\ln(p_{I,1}(x))+w_0\ln(1-p_{I,1}(x)).
\end{eqnarray}

At last, we want to say that  general
 image segmentation models can be got by taking away the non-negative constraint
for the SDF Laplacian.  For the sake of simplicity, we name the
GMM-based model, GMM-based model with landmarks and region priors based model with convexity prior as
GMMC, GMMLC and RPC,  and name the
corresponding models without convex shape prior as GMM, GMML and RP.
We will compare the results of the models with convexity prior and these
without convexity prior in Section \ref{sec5}.


\section{Algorithms for the proposed models}\label{sec4}
It is possible to design efficient algorithms for solving (\ref{modconv}) and (\ref{mod_LE}).
{
We concentrate on the method for $\phi$ update by fixing other variables,
and the parameters $c_i,\mu_i$ and $\Sigma_i$ are easy to estimate with given $\phi$.
}In this work, we consider a special case, but this special case
is general enough to handle most of the cases we encounter in practice. We shall assume that the segmented object is inside a domain $\Omega_1 \subset \Omega$. In case that our segmented object touches the boundary $\partial \Omega$, then it is possible to pad extra pixels around the image domain $\Omega$ and still use the algorithm.

In case that we know the segmented object is inside $\Omega_1 \subset \Omega$, we will only require that $\phi$ is
a signed distance function and satisfies $\Delta\phi  \ge 0$ inside $\Omega_1$, i.e.
\begin{equation}\label{eq:new1}  |\nabla \phi | = 1, \quad \Delta \phi \ge 0, \mbox{  in   } \Omega_1 . \end{equation}
In the rest of the domain $\Omega_2 = \Omega\backslash\Omega_1$, we will impose no constraint,
but require the function $\phi$ to satisfy the boundary condition
\begin{eqnarray}\label{eq:new2}
\frac {\partial \phi}{\partial \vec{n}} = \frac {\partial \Delta \phi}{\partial \vec{n}} = 0,  \mbox{  on } \partial \Omega .  \end{eqnarray}
Here and latter, $\vec{n}$ is the unit out normal vector of $\partial \Omega$.
In this section, we propose an efficient algorithm for minimization problems
(\ref{modconv}) and (\ref{mod_LE}) by using splitting technique properly.
%

By introducing two auxiliary variables $\zeta$ and $\xi$, the minimization problems  (\ref{modconv}) and (\ref{mod_LE})  under conditions  (\ref{eq:new1})
and (\ref{eq:new2}) are  equivalent to:
\begin{eqnarray}
\arg\min\limits_{\phi,\zeta,\xi} \int_\Omega F(\phi)dx, ~~\zeta(x) =\triangle\phi(x), \xi(x)=\nabla\phi(x)\notag\\
~~\zeta(x)\geq0,|\xi(x)|=1,  x\in\Omega_1, \frac {\partial \phi}{\partial\vec{n}} = \frac {\partial \Delta \phi}{\partial\vec{n}} = 0,  \mbox{  on } \partial \Omega ,
\label{modconvA}
\end{eqnarray}
where the integral function $F$ may be $F$ in (\ref{modconv}) or $F_L$ in (\ref{mod_LE}).
Three functional spaces are introduce for the convenience of narration as following:
\begin{eqnarray}
V&=&\{\phi\in H^2(\Omega)| ~~ \frac {\partial \phi}{\partial\vec{n}} = \frac {\partial \Delta \phi}{\partial\vec{n}} = 0  \mbox{  on } \partial \Omega \},\notag\\
V_1&=&\{\zeta\in H^1(\Omega)| ~~ {\partial\zeta\over \partial\vec{n}}=0 ~\mbox{on}~ \partial\Omega\},\nonumber\\
V_2&=&\{\xi\in H^1(\Omega)\times H^1(\Omega)| ~~ \xi\cdot\vec{n} = 0 ~\mbox{on}~ \partial\Omega   \}.\nonumber
\end{eqnarray}
The augmented Lagrangian functional of problem (\ref{modconvA}) is
\begin{eqnarray}
\begin{array}{l}
L(\phi,\xi,\zeta,\gamma_1,\gamma_2)=\int_\Omega{F}(\phi)dx+
\langle{\gamma_1,\triangle\phi-\zeta}\rangle\\
~~~+\langle{\gamma_2,\nabla\phi-\xi}\rangle+{\rho_1\over2}\|\triangle\phi-\zeta\|_2^2
+{\rho_2\over2}\|\nabla\phi-\xi\|_2^2,\\
~~~\text{subject ~to}~\zeta(x)\geq0,|\xi(x)|=1,x\in \Omega_1,
\end{array}\label{modLg}
\end{eqnarray}
where $\phi\in V,\gamma_1,\zeta \in V_1,\gamma_2,\xi\in V_2$, and  $\rho_1,\rho_2>0$ are two parameters.
In the following, we introduce two subsets of $V_1$ and $V_2$
\begin{eqnarray}
S_{11}&=&\{\zeta\in V_1| \zeta(x)\geq0,x\in\Omega_1\},\nonumber\\
S_{21}&=&\{\xi\in V_2|~|\xi(x)|=1,x\in\Omega_1\}.\nonumber
\end{eqnarray}
It is obvious that $\zeta\in S_{11}$ and $\xi\in S_{21}$.
In the formula (\ref{modLg}) and later, we use $\langle \cdot, \cdot \rangle$ to denote the $L^2$ inner product of functions.

Alternate direction method is applied to
minimize this problem iteratively. During the iterations, one of the three variables is updated  by fixing the others.
\subsection{Alternate direction method}
The alternating direction method of multipliers (ADMM)  for (\ref{modLg}) is given in Algorithm \ref{AlgGenal}.
\begin{algorithm}
\caption{Alternate direction algorithm for (\ref{modLg})}\label{AlgGenal}
\begin{algorithmic}
\STATE1. Initialization: $\gamma_1=0,\gamma_2=0,\rho_1,\rho_2>0$, initial curve\\~~$C$ or SDF $\phi^0$ and priors (landmarks or region prior, if needed), and total iteration number  $\text{Num}>0$;
\STATE2. Compute probabilities $p_{I,0}$ $p_{I,1}$ by (\ref{GMM}) and (\ref{eqLprb}) for different methods;
\STATE3. For $t=0,1,2,\cdots,\text{Num}$
\STATE4. ~~$\zeta^{t+1}=\arg\min_{\zeta(x)\in{S}_{11}}L(\phi^{t},\xi^{t},\zeta,\gamma^{t}_1,\gamma_2^{t})$,
\STATE5. ~~$\xi^{t+1}=\arg\min_{\xi\in{S}_{21}} L(\phi^{t},\xi,\zeta^{t+1},\gamma_1^{t},\gamma_2^{t})$,
\STATE6. ~~$\phi^{t+1}=\arg\min_{\phi\in V}L(\phi,\xi^{t+1},\zeta^{t+1},\gamma_1^{t},\gamma_2^{t})$,
\STATE7. ~~$\gamma^{t+1}_1=\gamma_1^{t}+\rho_1(\triangle\phi^{t+1}-\zeta^{t+1})$,
\STATE8. ~~$\gamma^{t+1}_2=\gamma_2^{t}+\rho_2(\nabla\phi^{t+1}-\xi^{t+1})$,
\STATE9. end(for)
\end{algorithmic}
\end{algorithm}

In the following, we give the details to solve the subproblems (Step 4,5,6) in Algorithm \ref{AlgGenal}.
There are closed form solutions for Step 4 and Step 5, and the
minimizer of step 6 is the solution of a fourth order partial differential equations with given boundary conditions.
\subsubsection {$\zeta$ update in Step 4} By discarding the terms independent of $\zeta$, we have that
$\zeta^{t+1}$ is the minimizer of the following problem:
\begin{eqnarray}
\arg\min_{\zeta\in S_{11}}L(\phi^{t},\xi^{t},\zeta,\gamma_1^{t},\gamma_2^{t})
=\arg\min\limits_{\zeta\in S_{11}}\int_{\Omega}
[{\rho_1\over2}(\zeta(x))^2-
\zeta(x)(\rho_1\triangle\phi^{t}(x)+\gamma_1^{t}(x))]dx.\nonumber
\end{eqnarray}
The minimizer of the above problem is
\begin{eqnarray}
\zeta^{t+1}(x)=\left\{\begin{array}{ll}
\max\{0,\widetilde{\zeta}^{t}(x)\}&x\in\Omega_1\\
\widetilde{\zeta}^t&x\in\Omega_2,
\end{array}
\right.\label{eqUpdateZ}
\end{eqnarray}
where $\widetilde{\zeta}^{t}(x)={\triangle\phi^{t}(x)+\gamma_1^{t}(x)/\rho_1}$.
\subsubsection{$\xi$ update in Step 5} Similarly, by removing the
terms independent of $\xi$,  we have that $\xi^{t+1}$ is the minimizer of
\begin{eqnarray}
\arg\min_{\xi\in S_{21}}L(\phi^{t},\xi,\zeta^{t+1},\gamma_1^{t},\gamma_2^{t})
&=&\arg\min_{\xi\in S_{21}}\int_{\Omega}[{\rho_2\over2}|\xi|^2-\xi\cdot(\rho_2\nabla\phi^{t}+\gamma_2^{t})]dx.
\end{eqnarray}
It is obvious that
\begin{eqnarray}
\xi^{t+1}(x)=
\left\{\begin{array}{ll}
\widetilde{\xi}^t(x)/|\widetilde{\xi}^t(x)|&x\in\Omega_1\\
\widetilde{\xi}^t(x)&x\in\Omega_2,
\end{array}
\right.
\label{pUpdate}
\end{eqnarray}
where $\widetilde{{\xi}^t}(x)=\nabla\phi^{t}(x)+\gamma_2^{t}(x)/\rho_2$.
\subsubsection{ $\phi$ update in Step 6}
In order to make the iteration procedure stable in Algorithm \ref{AlgGenal}, we add a proximity term to the augmented Lagrangian functional, i.e.
\begin{eqnarray}
\arg\min_\phi{L(\phi,\xi^{t+1},\zeta^{t+1},\gamma^{t}_1,\gamma_2^{t})}+
{\rho_0\over2}\|\phi-\phi^{t}\|_2^2
&=&\arg\min\int_{\Omega}[F(\phi)+{\rho_1\over2}(\triangle\phi)^2
+{\rho_2\over2}|\nabla\phi|^2]dx\nonumber\\
&&+\langle  \nabla\phi,\gamma_2^{t}-\rho_2\xi^{t+1}\rangle
+\langle \gamma_1^{t}-\rho_1\zeta^{t+1},\triangle\phi\rangle \label{eqphiF}\\
&&+{\rho_0\over2}\|\phi-\phi^{t}\|^2_2,\phi\in~V,\nonumber
\end{eqnarray}
where $\rho_0$ is a positive number. The above equality held  by  discarding independent terms.
Because $\phi \in V, \gamma_1^t,\zeta^{t+1}\in V_1,\gamma_2^t,\xi^{t+1}\in V_2$, we have
\[
{\partial\phi\over\partial\vec{n}}\!=\!
{\partial\triangle\phi\over\partial\vec{n}}\!=\!0, ~
{\partial\zeta^{t+1}\over\partial\vec{n}}\!={\partial\gamma_1^t\over\partial\vec{n}}\!=\!0,~
\xi^{t+1}\cdot \vec{n}\!=\! \gamma_2^t\cdot\vec{n}\!=\!0.
\]
Therefore, the Euler Lagrangian equation of the objective functional (\ref{eqphiF})   is
\begin{eqnarray}
\left\{\!\begin{array}{l}
\rho_1\triangle^2\phi^{t+1}-\rho_2\triangle\phi^{t+1}+F^\prime(\!\phi^{t\!+\!1}\!)
+\rho_0\!\phi^{t\!+\!1}\!=\!\text{rhd}^t~\text{in}~\Omega,\\
{\partial\phi^{t+1}\over\partial\vec{n}}=0,{\partial\triangle\phi^{t+1}\over\partial\vec{n}}=0,~~~\text{on}~~~
\partial\Omega,
\end{array}
\right.\notag
\end{eqnarray}
where $ \text{rhd}^t=\rho_0\phi^{t}-\triangle(\gamma_1^{t}-\rho_1\zeta^{t+1})-\nabla^T(\gamma_2^{t}
-\rho_2\xi^{t+1})$, and $\nabla^T$ denotes
the conjugate operator of $\nabla$.

Because $F$ is nonlinear, we solve the above equation iteratively as following
\begin{eqnarray}
\left\{\!
\begin{array}{l}
\rho_1\!\triangle^2\!\phi^{t,j\!+1}\!-\rho_2\!\triangle\phi^{t,j\!+\!1}\!+\!\rho_0\phi^{t,j\!+\!1}
\!=\!\text{RHD}^{t}(\phi^{t,j}),\\
{\partial\phi^{t+1}\over\partial\vec{n}}=0,{\partial\triangle\phi^{t+1}\over\partial\vec{n}}=0,~~~\text{on}~~~
\partial\Omega,
\end{array}
\right.
\end{eqnarray}
where $\text{RHD}^t(\phi^{t,j})=
\text{rhd}^t-F^\prime(\phi^{t,j})$ with initial value $\phi^{t,0}=\phi^t$.
Taking  $\rho_2=2\sqrt{\rho_0\rho_1}$, we can
reduce the above 4th order problem to the solution of two 2nd order problems. It is clear that
\begin{eqnarray}
\left\{
\begin{array}{l}
(\sqrt{\rho_1}\triangle-\sqrt{\rho_0}I)^2\phi^{t,j+1}= \text{RHD}^{t}(\phi^{t,j}),\\
{\partial\phi^{t+1}\over\partial\vec{n}}=0,{\partial\triangle\phi^{t+1}\over\partial\vec{n}}=0,~~~\text{on}~~~
\partial\Omega.
\end{array}\right.\label{phiUpdate}
\end{eqnarray}
We can obtain the solution to (\ref{phiUpdate})  by solving
two Laplacian equations
\begin{eqnarray}
&&\left\{\begin{array}{lcl}
(\sqrt{\rho_1}\triangle-\sqrt{\rho_0}I)\psi^{t,j+1}&=& \text{RHD}^{t}(\phi^{t,j})\label{eqL1}\\
{\partial \psi^{t,j+1}\over \partial\vec{n}}=0,
\end{array}
\right.\label{eqLapc1}\\
&&\left\{\begin{array}{lcl}
(\sqrt{\rho_1}\triangle-\sqrt{\rho_0}I)\phi^{t,j+1}&=& \psi^{t,j+1}\label{eqL1}\\
{\partial \phi^{t,j+1}\over \partial\vec{n}}=0.
\end{array}
\right.\label{eqLapc2}
\end{eqnarray}
The idea to reduce the 4th order partial differential equations into two 2nd order partial differential equations has been used for a one dimensional periodic boundary problem in \cite{Glowinski2013On}. Due to the special nature of our segmentation model, we are able to use this idea for our problem here and it removes the difficulties to solve 4th order partial differential equation and gives very good numerical efficiency for our proposed algorithm.
These two equations (\ref{eqLapc1}) and (\ref{eqLapc2}) can be solved by fast discrete cosine transform numerically \cite{Strange1999DCT}, which will be explained in detail in next subsection.

As for the selection of $\rho_0$, we follow the guideline below.
Let $A=\rho_1\triangle^2-\rho_2\triangle$. By the iteration formula, we have
\begin{eqnarray}
(A+\rho_0I)(\phi^{t,j+1}-\phi^{t,j})=F^\prime(\phi^{t,j})-
F^\prime(\phi^{t,j-1}).
\end{eqnarray}
Therefore, we have
\begin{eqnarray}
&&\|\phi^{t,j+1}-\phi^{t,j}\|_2\nonumber\\
&=&\|(A+\rho_0I)^{-1}F^\prime(\phi^{t,j})-F^\prime(\phi^{t,j-1})\|_2\nonumber\\
&\leq&\|(A+\rho_0I)\|_2^{-1}\max_x|F^{\prime\prime}(\phi(x))|\|
\phi^{t,j}-\phi^{t,j-1}\|_2.\nonumber
\end{eqnarray}
If $\|(A+\rho_0I)\|_2^{-1}\max_x|F^{\prime\prime}(\phi)|<1$, we have that the above iteration sequence is contracted,
and the sequence $\{\phi^{t,j}\}$ converges to $\phi^{t+1}$.
Because $A=\rho_1\triangle^2-\rho_2\triangle$ is a symmetric semidefinite operator, the iteration is contracted when $\rho_0>\max_x|F^{\prime\prime}(\phi)|$. In this paper, we use $\phi^{t,1}$ as the approximation of $\phi^{t+1}$ to save computational cost.

\subsection{Numerical implementation}
In this subsection, we give the details for the numerical implementations for different models in Section \ref{sec3}.
Firstly,  we use $H_\epsilon(s)$ and $\delta_\epsilon(s)=H^\prime_\epsilon(s)$ with $\epsilon>0$ to approximate the Heaviside function and Dirac distribution function
\begin{eqnarray}
\begin{array}{rcl}
 H_\epsilon(s)&=&{1\over2}+{1\over\pi}\arctan({s/\epsilon}),\\
\delta_\epsilon(s)&=&{\epsilon\over\epsilon^2+s^2}.
\end{array}
\end{eqnarray}

In practical applications, an input digital image $I\in\mathbb{R}^{M\times N\times d}$ ($d=1$ for gray image and $d=3$ for color image) is viewed as a discrete version of a continuous
image $I(x),x\in\Omega$ with mesh size $h=1$. By abusing the notations little, we also use
$\Omega=\{(m,n)|~ m=1,2\cdots,M,n=1,2,\cdots,N\}$ to denote the image domain, and $\Omega_1\subset \Omega$ is the given domain which contains the segmented object.
After given an initial curve in the image domain, we can compute the corresponding signed distance function $\phi$ in the image domain by the fast marching method \cite{Sethian1996A} or the fast sweeping method \cite{Tsai2004Fast,Zhao2005A}.
The operators $\triangle\phi$ and $\nabla\phi$ are approximated by finite differences.

For any function $\psi \in V$, we extend the discrete function $\psi(m,n)$ by one grid point around the image domain to satisfy the boundary condition ${\partial\psi\over\partial\vec{n}}=0$ with the discretization
\begin{eqnarray}
\psi(0,n)=\psi(1,n),&\psi(M+1,n)=\psi(M,n),\\
\psi(m,0)=\psi(m,1),&\psi(m,N+1)=\psi(m,N).
\end{eqnarray}
where $1\leq n\leq N, ~1\leq m\leq M$.
Because of the extension above, the differences $\nabla_x^{+}\psi(m,n)$ and $\nabla_y^{+}\psi(m,n)$ for $1\leq m\leq M,1\leq n\leq N$ are defined as
\begin{eqnarray}
\nabla^{+}_{x}\psi(m,n)&=&\psi_(m+1,n)-\psi(m,n),\\ \nabla_{y}^+\psi(m,n)&=&\psi(m,n+1)-\psi(m,n).
\end{eqnarray}
Then we  can numerically compute $\nabla{\psi}$ by
\begin{eqnarray}
\nabla{\psi}(m,n)&=&(\nabla_{x}^+\psi(m,n),\nabla_{y}^+\psi(m,n))^T.
\end{eqnarray}
Similarly, we can approximate $\triangle\psi(m,n)$ for $1\leq m\leq M,1\leq n\leq N$ by
central differences
\begin{eqnarray}
(\triangle\psi)(m,n)=(\nabla_{x}^{2}\psi)(m,n)+(\nabla_{y}^2\psi)(m,n),
\end{eqnarray}
where
\begin{eqnarray}
(\nabla_{x}^2\psi)(m,n)
=\psi(m+1,n)-2\psi(m,n)+\psi(m-1,n),\notag\\
(\nabla_{y}^2\psi)(m,n)
=\psi(m,n+1)-2\psi(m,n)+\psi(m,n-1).\nonumber
\end{eqnarray}
By the definition of conjugate operator, we have  $\langle\nabla\psi,q\rangle=\langle \psi,\nabla^Tq\rangle$ for
$q=(q_1,q_2)^T$, and
\begin{eqnarray}
(\nabla^Tq)(m,n)&=&-(\nabla_{x}^-q_1(m,n)+\nabla_{y}^-q_2(m,n)),\nonumber
\end{eqnarray}
where for $1\leq n\leq N$,
\[
\nabla^{-}_xq_1(m,n)=\left\{
\begin{array}{ll}
q_1(m,n)-q_1(m-1,n)&1<m<M\\
q_1(1,n)&m=1\\
-q_1(M-1,n)&m=M,
\end{array}
\right.
\]
and for $1\leq m\leq M$
\[
\nabla^{-}_yq_2(m,n)=
\left\{
\begin{array}{ll}
q_2(m,n)-q_2(m,n-1)&1<n<N\\
q_2(m,1)&n=1\\
-q_2(m,N-1)&n=N.
\end{array}
\right.
\]
In the following, we mainly discuss the
numerical formulas for the update of the  functions $\zeta,\xi$ and $\phi$ in Algorithm \ref{AlgGenal}.
For the implementation of $\zeta$ update formula (\ref{eqUpdateZ}),
we have
\begin{eqnarray}
\zeta^{t+1}(m,n)=
\left\{\begin{array}{ll}
\max\{\widetilde{\zeta}^t(m,n),0\}& (m,n)\in\Omega_1\\
\widetilde{\zeta}^t(m,n)&(m,n)\in\Omega\setminus\Omega_1,
\end{array}\right.\label{ztaUpdateD}
\end{eqnarray}
where $\widetilde{\zeta}^t(m,n)={\triangle\phi^{t}(m,n)+\gamma_1^{t}(m,n)/\rho_1}$.
By (\ref{pUpdate}),  we have
\begin{eqnarray}
\xi^{t+1}({m,n})=
\left\{\begin{array}{ll}
{\widetilde{\xi}^t(m,n)/|\widetilde{\xi}^t(m,n)|}&(m,n)\in \Omega_1\\
\widetilde{\xi}^t(m,n)&(m,n)\in \Omega\setminus\Omega_1,
\end{array}\right.\label{xiUpdateD}
\end{eqnarray}
where $\widetilde{\xi}^t(m,n)=\nabla\phi^{t}(m,n)+\gamma_2^{t}(m,n)/\rho_2$.
We will discuss the DCT method for the update of $\phi$.
Firstly, we can compute $\text{rhd}^t$ numerically,
\begin{eqnarray}
\text{rhd}^{t}(m,n)&=&\rho_0\phi^t(m,n)
-\triangle(\gamma_1^{t}(m,n)-\rho_1\zeta^{t+1}(m,n))\nonumber\\
&&- \nabla^T(\gamma_2^{t}(m,n)-\rho_2\xi^{t+1}(m,n)).\notag
\end{eqnarray}

We will present the details on the application of DCT to
solve Laplacian equation with Neumann boundary condition. The definition of DCT and some
properties used in this paper are presented first for one dimension case.
For discrete signal $S(j)~(j=1,2,\cdots J)$, its DCT transform is defined as
\begin{eqnarray}
\hat{S}(k)=w(k)\sum_{j=1}^{J}S(j)\cos({\pi(2j-1)(k-1)\over2J}),
\end{eqnarray}
where $~k=1,2,\cdots,J,~w(k)=\sqrt{1\over J}$ for $k=1$ and $\sqrt{2\over J}$ otherwise. Hereafter, we use $\hat{\bullet}$ to denote the discrete cosine transform of $\bullet$.

Let $e_J(j,k)=\cos\left({\pi(2j-1)(k-1)\over2J}\right)$. We have the following relationships between $e_J(j-1,k),~e_J(j,k)$ and $e_J(j+1,k)$ by trigonometric function formulae.
Firstly, for $1<j<J$ we have
\begin{eqnarray}
e_J(j-1,k)+e_J(j+1,k)
=2\cos({\pi(k-1)\over J})e_J(j,k). \label{equality1}
\end{eqnarray}
Secondly, we have
\begin{eqnarray}
&&e_J(1,k)+e_J(2,k)\nonumber\\
&=&\cos(\frac{k-1}{2J}\pi)+\cos(\frac{3(k-1)}{2J}\pi)\nonumber\\
&=&\cos(\frac{k-1}{J}\pi-\frac{k-1}{2J}\pi)+\cos(\frac{k-1}{J}\pi+\frac{k-1}{2J}\pi)\nonumber\\
&=&2\cos(\frac{k-1}{J}\pi)\cos(\frac{k-1}{2{J}}\pi)\nonumber\\
&=&2\cos(\frac{k-1}{J}\pi)e_J(1,k).\label{equality2}
\end{eqnarray}
Similarly, we can prove that
\begin{equation}
e_J(J,k)+e_J(J-1,k)=2\cos(\frac{(k-1)\pi}{J})e_J(J,k).
\label{equality3}
\end{equation}
By the equalities (\ref{equality1}), (\ref{equality2}) and (\ref{equality3}), we can get
\begin{eqnarray}
&&\widehat{\nabla_{x}^{2}S}(k)\notag\\
&=&w(k)\sum_{j=1}^J(\nabla_{x}^{2}S)(j)e_J(j,k)\nonumber\\
&=&w(k)\sum_{j=1}^J(S(j+1)-2S(j)+S(j-1))e_J(j,k)\nonumber\\
&=&-2\hat{S}(k)+w(k)S(J)e_J(J,k)+w(k)S(1)e_J(1,k)
+w(k)\sum_{j=2}^JS(j)e_J(j-1,k)+w(k)\sum_{j=1}^{J-1}S(j)e_J(j+1,k)\nonumber\\
&=&-2\hat{S}(k)+w(k)\sum_{j=2}^{J-1}S(j)[e_J(j-1,k)+e_J(j+1,k)]
+w(k)S(J)[e_J(J,k)+e_J(J-1,k)]\notag\\
&&+w(k)S(1)[e_J(2,k)+e_J(1,k)]\notag\\
&=&-2\hat{S}(k)+2w(k)\cos({\pi(k-1)\over{J}})\sum_{j=1}^JS(j)e_J(j,k)\notag\\
&=&2(\cos({\pi(k-1)\over{J}})-1)\hat{S}(k).\label{eq1DCT}
\end{eqnarray}

For two dimensional signal $\phi(m,n),m=1,2,\cdots,M,n=1,2,\cdots,N$, its 2D DCT is defined as
\begin{eqnarray}
\hat{\phi}(k,l)=w(k)w(l)\sum_{m=1,n=1}^{M,N}
\phi(m,n)e_M(m,k)e_N(n,l),
\end{eqnarray}
where $k=1,2,\cdots,M,l=1,2\cdots,N$.
Using the similar derivation for (\ref{eq1DCT}), we can get that for $k=1,2,\cdots,M,l=1,2\cdots,N$
\begin{eqnarray}
\widehat{\nabla_{x}^{2}\phi}(k,l)=
2(\cos({\pi(k-1)\over{M}})-1)\hat{\phi}(k,l),\\
\widehat{\nabla_{y}^2\phi}(k,l)=2(\cos({\pi(l-1)\over{N}})-1)
\hat{\phi}(k,l).
\end{eqnarray}

Applying discrete cosine transform \cite{Strange1999DCT} on both sides of (\ref{phiUpdate}), we have
\begin{eqnarray}
r(k,l)\widehat{\psi}^{t,j+1}(k,l)
&=&\widehat{\text{RHD}^{t}}(\phi^{t,j})(k,l),\\
r(k,l)\hat{\phi}^{t,j+1}(k,l)
&=&\widehat{\psi}^{t,j+1}(k,l),
\end{eqnarray}
where $r(k,l)=\sqrt{\rho_0}+2\sqrt{\rho_1}\left[2-\cos\left({k-1\over M}\pi\right)-\cos\left({l-1\over N}\pi\right)\right]$ for
$1\leq k\leq M,1\leq l\leq N$.
In summary, we have
\begin{eqnarray}
[r(k,l)]^2\hat{\phi}^{t,j+1}(k,l)=
\widehat{\text{RHD}^{t}}(\phi^{t,j})(k,l) . \label{phiUpdateD}
\end{eqnarray}
We want to emphasize that $\widehat{\text{RHD}^{t}}$ are different for different models.
For the GMM-based and region-prior models, we have for all $1\leq m\leq M,1\leq n\leq N$,
\begin{eqnarray}
\text{RHD}_1^t(\phi)(m,n)=\text{rhd}^t(m,n)-F_\epsilon(\phi^t(m,n)),
\end{eqnarray}
where
\begin{eqnarray}
F_\epsilon\!(\!\phi^t(m,n)\!)\!=\!\delta_\epsilon\!(\phi^t(m,n)\!)\!f(m,n)\!+\!
\delta_\epsilon^\prime\!(\phi^t(m,n)\!)g\!(m,n\!),\notag
\end{eqnarray}
and $f(m,n)=-w_1\ln(p_{I,1}(m,n))+w_0\ln(1-p_{I,1}(m,n))$, $g(m,n)={\alpha\over1+\beta|\nabla G\ast{I(m,n)}|}$.
For the GMM-based model with boundary landmarks $x_k=(x_{k,1},x_{k,2})(k=1,2,\cdots,K)$, we have
\begin{eqnarray}
\text{RHD}_2^t(\phi)(m,n)=\text{RHD}_1^t+\theta\phi^t(m,n)\sum_{k=1}^K\delta_{x_k}(m,n)\notag
\end{eqnarray}
where $\delta_{x_k}(m,n)\!=\!1$ if $m\!=\!x_{k,1},n\!=\!x_{k,2}$ and zero otherwise.

It is obvious that $r(k,l)>0$ for all $1\leq k\leq M,1\leq l\leq N$. Dividing $[r(k,l)]^2$ on both sides, and using 2D inverse DCT, we can get $\phi^{t,j+1}$.

We will present the concrete algorithms for GMM-based models with and without landmarks and
region-prior model as following.
After having the SDF $\phi^t$ by initial curve or updated
by the iteration procedure, we can estimate the parameters as following
\begin{eqnarray}
c_i^{t}&=&{1\over{MN}}\sum_{m,n}{q_i^t}({m,n}),i=0,1\label{cUpdateD}\\
\mu^{t}_i&=&{\sum_{m,n}{q_i^t(m,n)I(m,n)}\over\sum_{m,n}{q_i^t(m,n)}},i=0,1\label{muUpdateD}\\
\Sigma_i^{t}&=&{\sum_{m,n}{q_i^t(m,n)u^t(m,n)}\over\sum_{m,n}{q_i^t(m,n)}},i=0,1\label{sigmaUpdateD},
\end{eqnarray}
where $q_1^t=H_\epsilon(\phi^t),q_0^t=1-q_1^t$, and $u^t(m,n)=(I(m,n)-\mu_i^{t})^T(I(m,n)-\mu_i^{t})$.
Having the parameters $\mu_i^t,\Sigma^t_i~(i=0,1)$, we have the probabilities
\begin{eqnarray}
p_i^{t}(m,n)={1\over(2\pi)^{d\over2}\det(\Sigma_i^{t})^{1\over2}}
\exp(v^t_i(m,n)),i=0,1,\label{ProUpdateGMM}
\end{eqnarray}
where $v_i^t(m,n)={1\over2}(I(m,n)-\mu_i^{t})^T(\Sigma_i^{t})^{-1}(I(m,n)-\mu_i^{t})$.
In the implementation, we add a diagonal matrix $D=\text{diag}(\lambda,\lambda,\lambda)$ to $\Sigma_i^t$ to avoid $\Sigma_i^t$
rank deficiency ($\lambda=0.1$ in this paper).
Finally, we can compute the probabilities $p_{I,i}(i=0,1)$ for all pixels $(m,n)~(1\leq m\leq M,~1\leq n\leq N)$
\begin{eqnarray}
p_{I,i}^{t}(m,n)
&=&{c_i^{t}p_i^{t}(m,n)
\over{c_0^tp_0^t(m,n)+c_1^{t}p_1^{t}(m,n)}},i=0,1,
\label{GMMD}
\end{eqnarray}
and obtain the new region force terms by (\ref{eq:RegGMM}) for the image GMM-based segmentation models numerically.

Based on the discussions above, we give the
algorithm for the GMM-based models with and without landmarks on the object boundary as Algorithm \ref{AlgGMM}.

\begin{algorithm}
\caption{Algorithm for the GMM-based model}\label{AlgGMM}
\begin{algorithmic}
\STATE1. Initialization: $\zeta=0\in\mathbb{R}^{M,N},
\gamma_1=0\in\mathbb{R}^{M,N}$, $\xi=0\in\mathbb{R}^{M,N}\times\mathbb{R}^{M,N},
\gamma_2=0\in\mathbb{R}^{M,N}\times\mathbb{R}^{M,N}$,
$\rho_1,\rho_0>0$, $\rho_2=\sqrt{\rho_1\rho_0}$, $Num>0$, initial curve $C$ and
landmarks $x_k(k=1,2,\cdots,K)$(if needed)
\STATE2. Compute the SDF $\phi^0$ of $C$, and estimate (initial)
probabilities $p^0_{I,i}(x),i=0,1$ by (\ref{GMMD})
\STATE3. For~ $t=0,1,2,\cdots,Num$
\STATE4. $\zeta^{t+1}$ update by (\ref{ztaUpdateD}).
\STATE5.
$\xi^{t+1}$ update by (\ref{xiUpdateD}).
\STATE6.
$\phi^{t+1}$ update by (\ref{phiUpdateD}) and DCT.
\STATE7. $\gamma_1^{t+1}(m,n)\!=\!\gamma_1^{t}(m,n)\!+\!\rho_1(\triangle\phi^{t\!+\!1}(m,n)\!-\!\zeta^{t\!+\!1}(m,n)).$
\STATE8. $\gamma_2^{t\!+\!1}(m,n)\!=\!\gamma_2^{t}(m,n)\!+\!\rho_2(\nabla\phi^{t\!+\!1}(m,n)\!-\!\xi^{t+1}(m,n)).$
\STATE9. Update $c_i^{t},\mu_i^{t}$ and $\Sigma_i^{t}(i=0,1)$ by (\ref{cUpdateD}), (\ref{muUpdateD}) and (\ref{sigmaUpdateD}), and compute new probabilities by (\ref{GMMD}).
\STATE10. End(for)
\end{algorithmic}
\end{algorithm}

\begin{remark} We want to remark that the boundary conditions for $\xi,\zeta, \gamma_i, i=1,2$ are automatically satisfied if the initial values satisfy these boundary conditions. Thus, the values for these functions are not updated for the grid points that are added around the image domain.
\end{remark}
As for the region prior method,
let $R_{ob}$ and $R_{bg}$ be the two labelled region prior sets of object and background.
The probabilities are computed by
\begin{eqnarray}
p_{I,1}(m,n)=
{
\sum_{(k,l)\in{R_{bg}}}\exp(-E(m,n,k,l))
\over
\sum_{(k,l)\in{R_{ob}\bigcup}R_{bg}}\exp(-E(m,n,k,l))}\label{CMLR}
\end{eqnarray}
and $p_{I,0}(m,n)=1-p_{I,1}(m,n)$, where \[E(m,n,k,l)\!=\!a_1(k-m)^2\!+\!a_1(l-n)^2\!+\!a_2\|I(m,n)\!-\!I(k,l)\|_2^2.\]
The numerical algorithm for the region-prior model is given in Algorithm \ref{AlgRp}.
\begin{algorithm}
\caption{Algorithm for the update of all variables}\label{AlgRp}
\begin{algorithmic}
\STATE1. Initialization: $\zeta=0\in\mathbb{R}^{M,N},
\gamma_1=0\in\mathbb{R}^{M,N}$, $\xi=0\in\mathbb{R}^{M,N}\times\mathbb{R}^{M,N},
\gamma_2=0\in\mathbb{R}^{M,N}\times\mathbb{R}^{M,N}$,
$\rho_1,\rho_0>0$, $\rho_2=\sqrt{\rho_1\rho_0}$, $Num>0$, initial curve $C$ and region priors $R_{bg}$ and $R_{ob}$;
\STATE2. Compute the SDF $\phi^0$ of $C$ , and estimate (initial)
probabilities $p^0_{I,i}(x),i=0,1$ by and (\ref{CMLR})
\STATE3. For~ $t=0,1,2,\cdots,Num$
\STATE4. $\zeta^{t+1}$ update by (\ref{ztaUpdateD})
\STATE5.
$\xi^{t+1}$ update by (\ref{xiUpdateD})
\STATE6.
$\phi^{t+1}$ update by (\ref{phiUpdateD}) and DCT
\STATE7. $\gamma_1^{t\!+\!1}(m,n)\!=\!\gamma_1^{t}(m,n)\!+\!\rho_1(\triangle\phi^{t+1}(m,n)\!-\!\zeta^{t+1}(m,n))$
\STATE8.$\gamma_2^{t\!+\!1}(m,n)\!=\!\gamma_2^{t}(m,n)\!+\!\rho_2(\nabla\phi^{t+1}(m,n)\!-\!\xi^{t+1}(m,n))$.
\STATE9. End(for)
\end{algorithmic}
\end{algorithm}

We want to emphasize that it is easy to use GMM, GMML and RP to segment shapes without convexity requirement. This can be easily done by removing the nonnegativity projection for the update of $\zeta$ in Step 4.

\section{Numerical Experiments}\label{sec5}
In this section, we will present some numerical examples to show the
efficiency of the proposed methods.
Numerous experiments by the proposed methods were conducted.
The experimental images include complex binary images,
occluded images and low contrast images. It is  challenging to segment these  images.
Part of our results are categorized and presented in following.
The results, especially on the complex real images,  show the models with convex shape prior
is superior to the models without convex shape prior.

Some parameters are kept the same in the experiments for all images.
In the computation of edge detection function $g$, $\alpha=0.1,\beta=10$,
$G={1\over16}[1,2,1;2,4,2;1,2,1]$ in (\ref{eqG}), and the gradient is approximated by the Sobel operator.
In the implementation,  $\rho_0=10$, $\rho_1=1$ and $\rho_2=2\sqrt{10}$.
In all the experiments, we set $\Omega_1=\{(m,n)|1<m<M,1<n<N\}$ for $I\in \mathbb{R}^{M\times N\times d}$
($d=1$ for gray image and $d=3$ for color image). In fact, this is the largest possible choice for $\Omega_1$ in the discrete setting.
%
\subsection{Comparisons between GMM and GMMC}
In this subsection we present some experimental results
by GMM  and GMMC models.
The experimental images include convex and nonconvex objects.
The segmentation results show the validity of  GMMC and the proposed
algorithm in keeping convexity of object.
The numerical results by GMMC and GMM are displayed in Figure
\ref{FigGMM1} and \ref{FigGMM2}.
The parameters $w_0,w_1$ are tabulated in Table \ref{TGMM} for different images.

Figure \ref{FigGMM1} illustrates the segmentation results of images with nonconvex objects.
Firstly, these results show the correctness of the
proposed algorithms for the models.
Secondly, the results by GMMC show the validity of GMMC in keeping the convexity of objects.
Although the horse, teapot, gear and Chinese fan are not convex shapes, GMMC can output meaningful convex
contours surrounding the objects under proper parameters $w_0,w_1$.
\begin{figure}[htb]
\centering
\begin{tabular}{ccc}
\includegraphics[width=3cm]{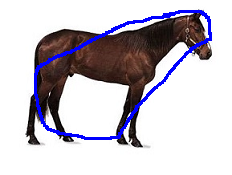}&
\includegraphics[width=2.5cm]{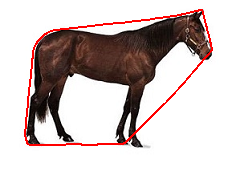}&
\includegraphics[width=2.5cm]{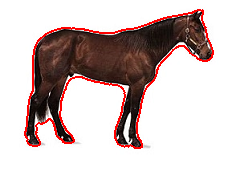}\\
\includegraphics[width=2.5cm]{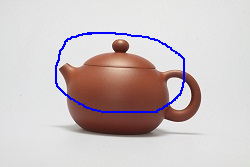}&
\includegraphics[width=2.5cm]{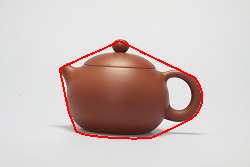}&
\includegraphics[width=2.5cm]{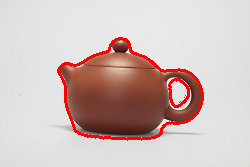}\\
\includegraphics[width=2.5cm]{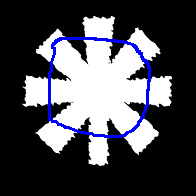}&
\includegraphics[width=2.5cm]{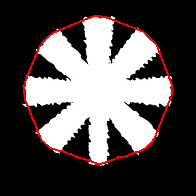}&
\includegraphics[width=2.5cm]{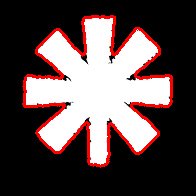}\\
\includegraphics[width=2.5cm]{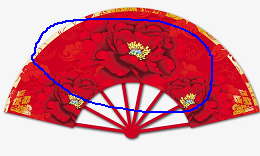}&
\includegraphics[width=2.5cm]{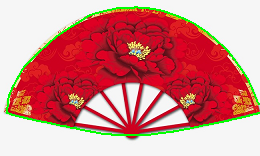}&
\includegraphics[width=2.5cm]{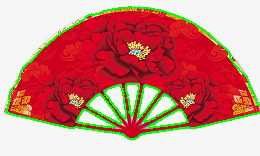}\\
\includegraphics[width=2.5cm]{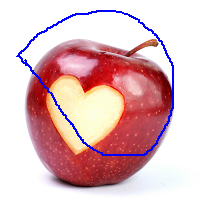}&
\includegraphics[width=2.5cm]{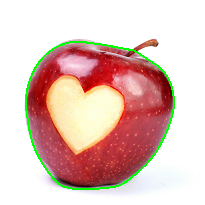}&
\includegraphics[width=2.5cm]{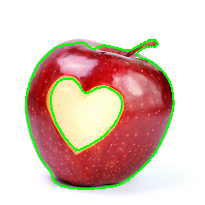}
\end{tabular}
\caption{Segmentation results by GMMC and GMM for images with nonconvex objects. First column: Original images with initial curves; Second column: Segmentation results by GMMC; Third column: Segmentation results by GMM.}\label{FigGMM1}
\end{figure}

Figure \ref{FigGMM2} illustrates some results on images with  occlusions, such as the Chinese fan, jade, bitten leaf, football and black bear.
It is more easy to identify the segmentation results by GMMC than these by
GMM for computer vision.
These results show that GMMC can touch the whole object contour accurately, while  GMM  fails to get meaningful segmentations (see the bear image for example).

\begin{figure}[htb]
\centering
\begin{tabular}{ccc}
\includegraphics[width=2.5cm]{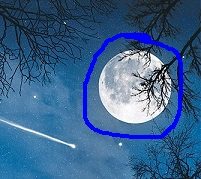}&
\includegraphics[width=2.5cm]{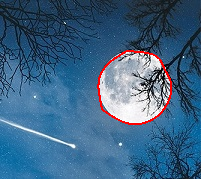}&
\includegraphics[width=2.5cm]{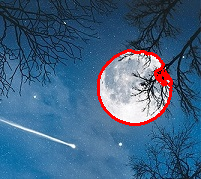}\\
\includegraphics[width=2.5cm]{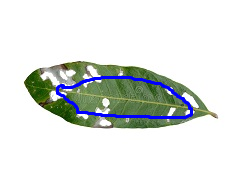}&
\includegraphics[width=2.5cm]{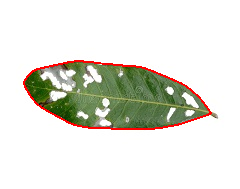}&
\includegraphics[width=2.5cm]{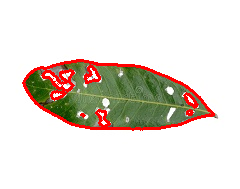}\\
\includegraphics[width=2.5cm]{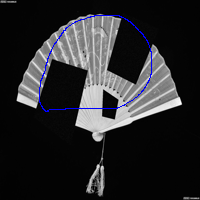}&
\includegraphics[width=2.5cm]{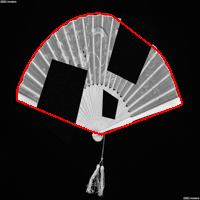}&
\includegraphics[width=2.5cm]{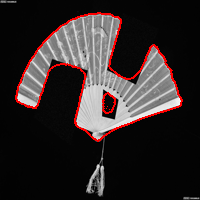}\\
\includegraphics[width=2.5cm]{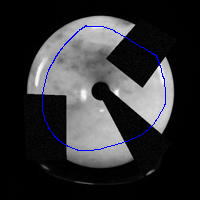}&
\includegraphics[width=2.5cm]{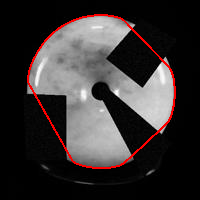}&
\includegraphics[width=2.5cm]{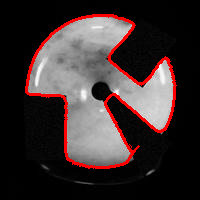}\\
\includegraphics[width=2.5cm]{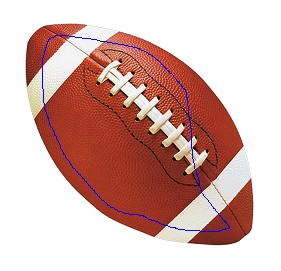}&
\includegraphics[width=2.5cm]{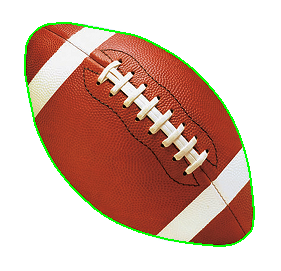}&
\includegraphics[width=2.5cm]{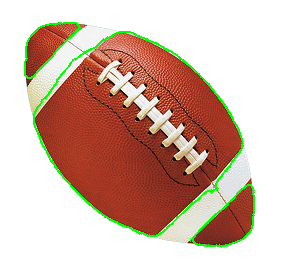}\\
\includegraphics[width=2.5cm]{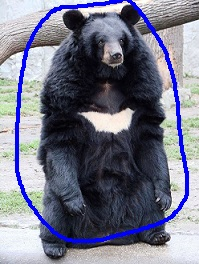}&
\includegraphics[width=2.5cm]{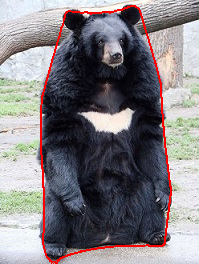}&
\includegraphics[width=2.5cm]{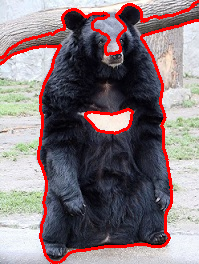}
\end{tabular}
\caption{Results by GMMC and GMM for occluded and low contrast images.
First column: Original images with initial curves; Second column: Segmentation results by GMMC; Third column: Segmentation results by GMM.}\label{FigGMM2}
\end{figure}

\begin{table}[htb]
\centering
\caption{Parameters for the images in Figure \ref{FigGMM1} and \ref{FigGMM2} }\label{TGMM}
\begin{tabular}{ccccccccc}
\hline\hline
image     &horse    &moon     &gear   &leaf   &occlusion fan    &football \\
\hline
$[w_0,w_1]$&$[2,0.5]$&$[2,0.8]$&$[2,1]$&$[2,1]$&$[2,0.8]$&$[1,0.5]$\\
image &teapot   &   jade&bear&red fan&apple\\
$[w_0,w_1]$&$[2,0.4]$&$[2,1.5]$&$[2,1]$&$[2,1.5]$&$[2,1.5]$\\
\hline\hline
\end{tabular}
\end{table}

In order to investigate the effects of the parameters $w_0,w_1$,
experiments for the same images were conducted by GMMC and GMM with different parameters.
The results by three sets of parameters are illustrated in Figure \ref{FigGMM3}.  They show that GMM is
more sensitive than GMMC to parameter variations.
\begin{figure}[htb]
\centering
\begin{tabular}{ccc}
\includegraphics[width=2.5cm]{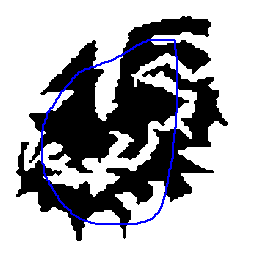}&
\includegraphics[width=2.5cm]{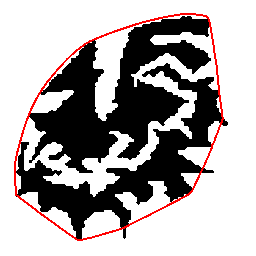}&
\includegraphics[width=2.5cm]{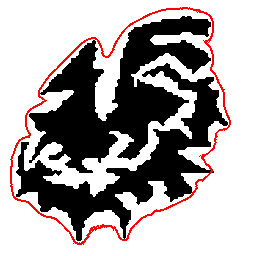}\\
\includegraphics[width=2.5cm]{image/GMM/img19/Curv}&
\includegraphics[width=2.5cm]{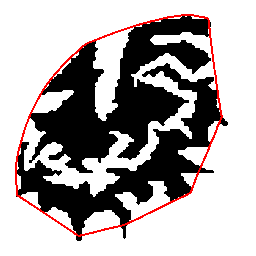}&
\includegraphics[width=2.5cm]{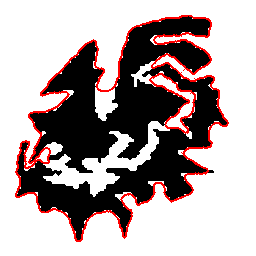}\\
\includegraphics[width=2.5cm]{image/GMM/img19/Curv}&
\includegraphics[width=2.5cm]{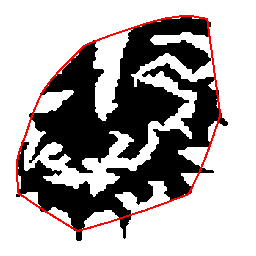}&
\includegraphics[width=2.5cm]{image/GMM/img19/exp1/nGseg}\\
\end{tabular}
\caption{Image segmentation with convex constraints. First column: Original images with initial curves; Second column: Segmentation results by GMMC; Third column: Segmentation results by GMM. The parameters $w_0,w_1$ for images in rows from up to bottom  are $[2,0.2]$, $[2,0.5]$ and $[2,1]$.}\label{FigGMM3}
\end{figure}
\subsection{Comparison between GMMLC and GMML}
In this subsection, we present some results by GMMLC and GMML to
show the efficiency of GMMLC for complex image segmentation.
These images can be categorized into two groups illustrated in Figure \ref{FigE1} and Figure \ref{FigE3}, respectively.
The first group includes images with nonuniform object intensities,
such as the ink painting fish, Yin-Yang, magic cubic.
The second group includes some low contrast images, such as
 the Pentagon,  medical image and colored stone.
The parameters $w_0$ $w_1$, and the number of landmarks are tabulated in Table \ref{TabE}. The landmarks are marked as red points in the original images
(first column in Figure \ref{FigE1} and \ref{FigE3}). The initial
curves for these images are tabulated in the second column of Figures \ref{FigE1} and \ref{FigE3}.
The segmentation results by GMMLC and GMML are displayed in the third and fourth columns of Figure \ref{FigE1} and
\ref{FigE3}.

The experiments show that the
convexity prior plays an important role for
GMMLC. By comparing the results in Figure \ref{FigE1} and Figure \ref{FigE3},
it is obvious that the given landmarks can
help GMMLC to get the correct shape boundary, but fails for GMML.
We observed that the landmarks helped to drag the zero level set of $\phi$ globally and gradually to the
object contour for GMMLC because the
convexity shape constraint do not allow the curve to split. However,
the landmarks can only affect the evolution of the zero level set curve locally for GMML.

Figure \ref{FigE1} shows some segmentation results
for images with nonuniform intensities in the object domain, such as the ink painting fish, Yin-Yang.
Because the intensities of the object varies dramatically, such as ink painting fish and the soccer ball,
 it is very challenging to get the correct object contour. With the help of landmarks on the boundary,
the results show that GMMLC can get the object boundary successfully.
However, the landmarks helps little for GMML, and we can't get meaningful segmentation results by GMML.

\begin{table}[htb]
\centering
\caption{Model parameters and landmark point number for images in \ref{FigE1} and \ref{FigE3}.}\label{TabE}
\begin{tabular}{ccccccccc}
\hline\hline
Image&Pentagon&magic cube&moon&medImag\\
\hline
\\
Landmarks&8&6&4&6
\\
$[w_0,w_1]$&$[1,1.5]$&$[1,0.5]$&$[1,0.5]$&$[1,0.1]$\\
\\
Image&ship&fish&QR code&building
\\
Landmarks&5&7&4&8
\\
$[w_0,w_1]$&$[1,0.8]$&$[1,0.1]$&$[1,0.8]$&$[1,0.4]$\\
Image&Yin-Yang&abacus& soccer&color stone
\\
Landmarks&8&5&5&4
\\
$[w_0,w_1]$&$[1,0.5]$&$[1,1]$&$[1,0.5]$&$[1,0.5]$\\
\hline\hline
\end{tabular}
\end{table}
\begin{figure}[htb]
\centering
\begin{tabular}{cccc}
\includegraphics[width=1.8cm]{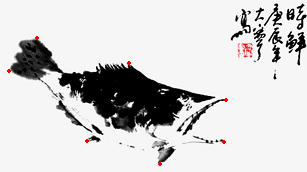}&
\includegraphics[width=1.8cm]{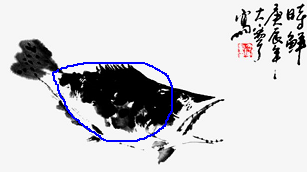}&
\includegraphics[width=1.8cm]{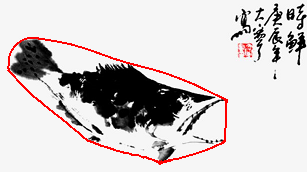}&
\includegraphics[width=1.8cm]{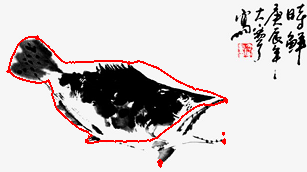}\\
\includegraphics[width=1.8cm]{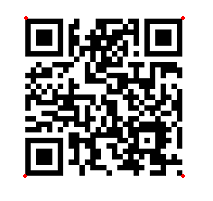}&
\includegraphics[width=1.8cm]{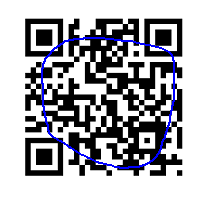}&
\includegraphics[width=1.8cm]{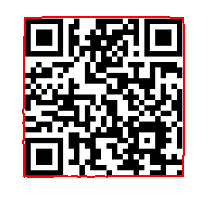}&
\includegraphics[width=1.8cm]{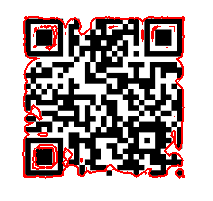}\\
\includegraphics[width=1.8cm]{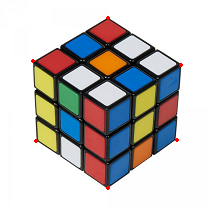}&
\includegraphics[width=1.8cm]{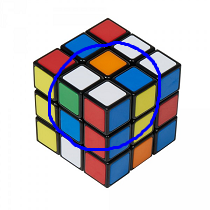}&
\includegraphics[width=1.8cm]{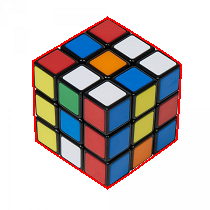}&
\includegraphics[width=1.8cm]{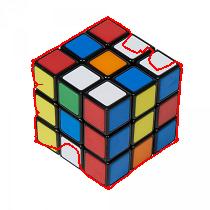}\\
\includegraphics[width=1.8cm]{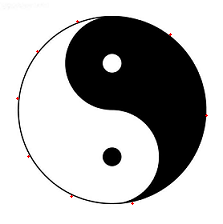}&
\includegraphics[width=1.8cm]{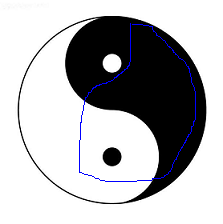}&
\includegraphics[width=1.8cm]{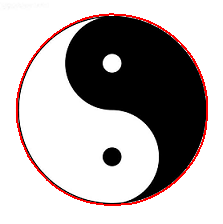}&
\includegraphics[width=1.8cm]{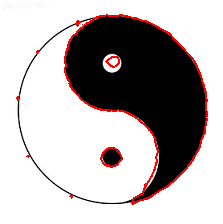}\\
\includegraphics[width=1.8cm]{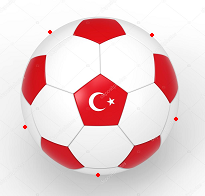}&
\includegraphics[width=1.8cm]{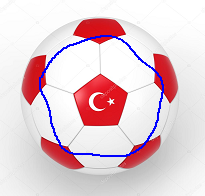}&
\includegraphics[width=1.8cm]{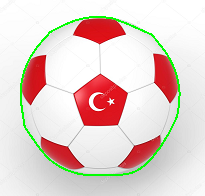}&
\includegraphics[width=1.8cm]{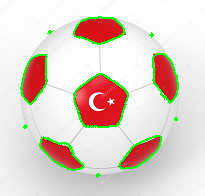}\\
\includegraphics[width=1.8cm]{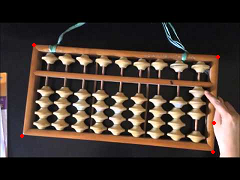}&
\includegraphics[width=1.8cm]{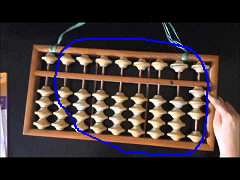}&
\includegraphics[width=1.8cm]{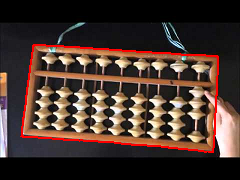}&
\includegraphics[width=1.8cm]{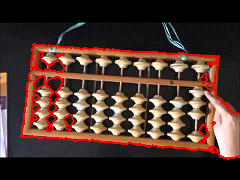}\\
\includegraphics[width=1.8cm]{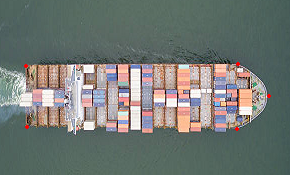}&
\includegraphics[width=1.8cm]{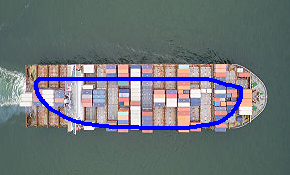}&
\includegraphics[width=1.8cm]{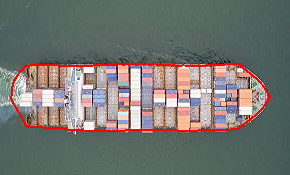}&
\includegraphics[width=1.8cm]{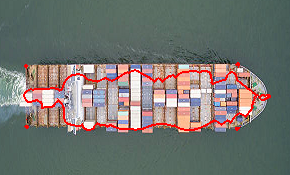}
\end{tabular}
\caption{Segmentation results for image with nonuniform intensity object by GMMLC and GMML. }\label{FigE1}
\end{figure}
Some experiments of low contrast images are presented in Figure \ref{FigE3}.
The intensities contrast near the object boundary is very low,
such as the boundary of colored stone and Pentagon. Therefore,
it is very hard to get the object boundary accurately. With the help of landmarks, the convexity shape model can
get the object contour successfully.
\begin{figure}[htb]
\centering
\begin{tabular}{cccc}
\includegraphics[width=1.8cm]{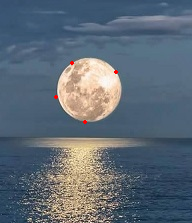}&
\includegraphics[width=1.8cm]{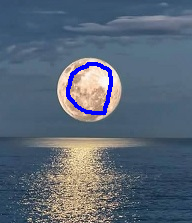}&
\includegraphics[width=1.8cm]{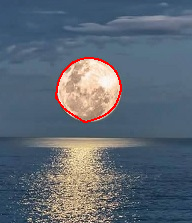}&
\includegraphics[width=1.8cm]{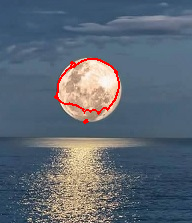}\\
\includegraphics[width=1.8cm]{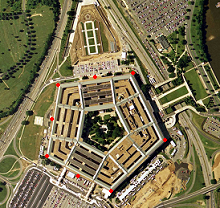}&
\includegraphics[width=1.8cm]{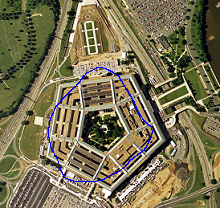}&
\includegraphics[width=1.8cm]{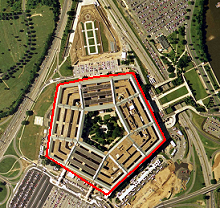}&
\includegraphics[width=1.8cm]{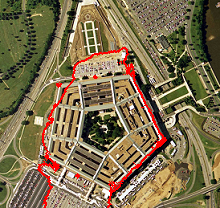}\\
\includegraphics[width=1.8cm]{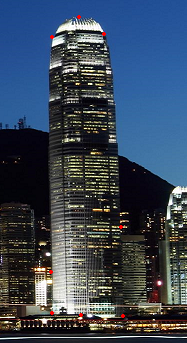}&
\includegraphics[width=1.8cm]{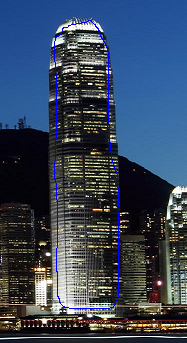}&
\includegraphics[width=1.8cm]{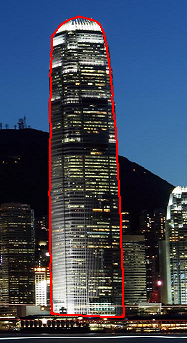}&
\includegraphics[width=1.8cm]{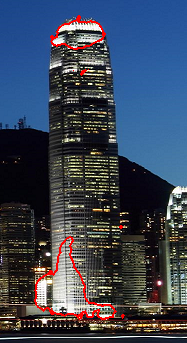}\\
\includegraphics[width=1.8cm]{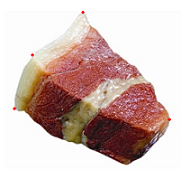}&
\includegraphics[width=1.8cm]{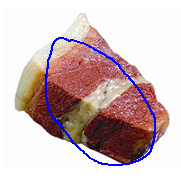}&
\includegraphics[width=1.8cm]{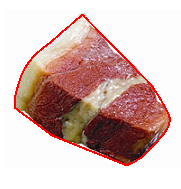}&
\includegraphics[width=1.8cm]{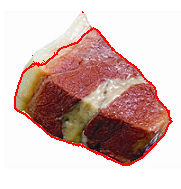}\\
\includegraphics[width=1.8cm]{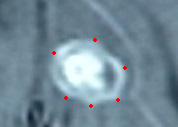}&
\includegraphics[width=1.8cm]{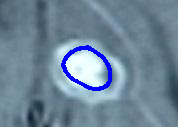}&
\includegraphics[width=1.8cm]{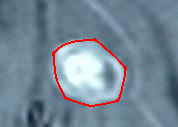}&
\includegraphics[width=1.8cm]{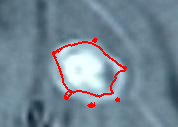}
\end{tabular}
\caption{Segmentation results for image with low contrast near boundary by GMMLC and GMML.}\label{FigE3}
\end{figure}

In the end, we want to say some thing about the initial curve for GMMLC.
Our experiences show that it is better to give an initial curve such that the given landmarks are outside it.
For this case, the iteration sequence normally have
high convergence rate.  We give an intuitive explanations for this observation.  On the one hand,  the penalty
term with respect to the given landmarks will make the zero level set (curve) pass through the landmark points,
i.e. some points on the zero level set curve move to the landmarks.
Therefore, the zero level set of $\phi$ to be concave gradually between two landmark points.
On the other hand, the convexity constraint will make all the curve move outer to keep it convex.
Conversely, if the initial curve is given such that the landmarks are inside it, the
 penalty term  would drag the curve to move in partially,
 which would make the SDF or zero level set curve nonconvex near the given landmarks. It is
 contradict to convexity constraint,  and results in the low convergence rate of iteration sequence.

\subsection{Comparison between RP and RPC}
In this section some numerical results by RPC and RP are
presented in Figure \ref{FigR}. Red parts of the images are the labelled points for the object and green points
are labelled points for the background (see the image in second column of Figure \ref{FigR}).
The parameters $w_0=w_1=1$ are used for all images.

The results in Figure \ref{FigR} show the high performance of RPC for complex and low contrast and occluded image segmentations.
We can see that RPC is able to give us correct and desirable segmentation results under the given labels.
However, the segmentation results by RP without convexity shape constraint is incorrect even though the same labels are given.
\begin{figure}[htb]
\centering
\begin{tabular}{ccccc}
\includegraphics[width=1.8cm]{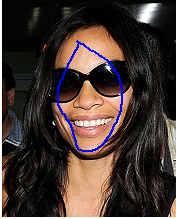}&
\includegraphics[width=1.8cm]{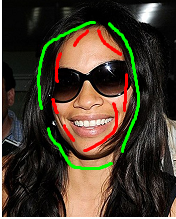}&
\includegraphics[width=1.8cm]{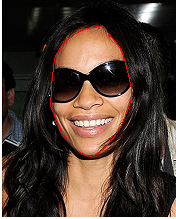}&
\includegraphics[width=1.8cm]{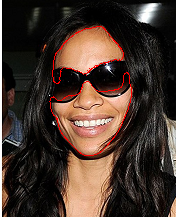}
\\
\includegraphics[width=1.8cm]{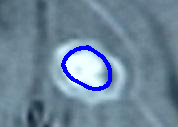}&
\includegraphics[width=1.8cm]{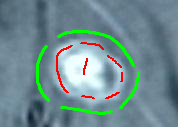}&
\includegraphics[width=1.8cm]{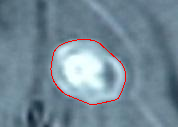}&
\includegraphics[width=1.8cm]{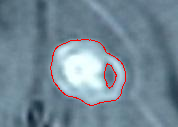}\\
\includegraphics[width=1.8cm]{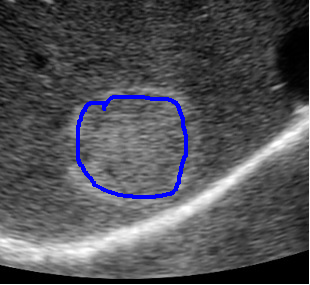}&
\includegraphics[width=1.8cm]{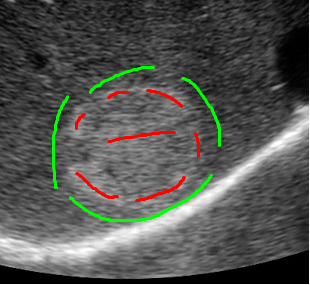}&
\includegraphics[width=1.8cm]{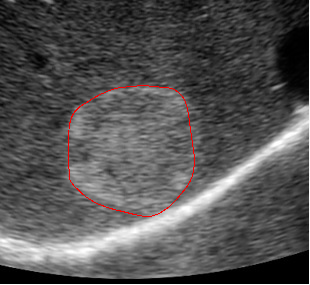}&
\includegraphics[width=1.8cm]{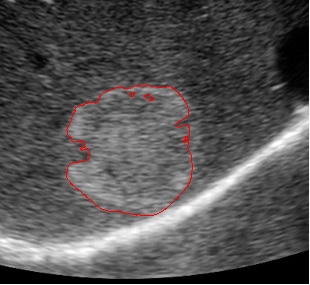}\\
\includegraphics[width=1.8cm]{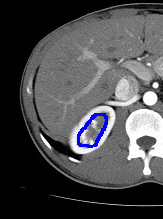}&
\includegraphics[width=1.8cm]{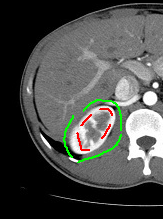}&
\includegraphics[width=1.8cm]{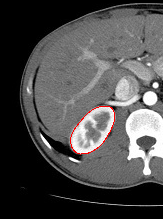}&
\includegraphics[width=1.8cm]{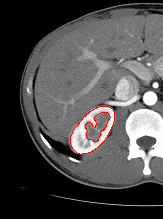}
\end{tabular}
\caption{Experiment results by RPC and RP without convexity constraint.
First column: Original images with
initial curves; Second column: Images with labels, red points for object and green points for background;
Third column: results by RPC; Fourth column: results by RP.}\label{FigR}
\end{figure}

\section{Conclusions and Future Works}\label{sec6}
Image segmentation with shape priors attracted more and more
attentions recently.
In this paper we proposed three models to implement two-phase image segmentation with convexity prior. Efficient algorithms by using splitting technique  were
developed for the proposed models.
The accuracy of the proposed methods were validated by numerical experiments on various complex, occluded, low contrast and nonuniform intensity images.
The results show that the proposed models with convex shape prior can keep the shape convexity and get more accurate results than the  models without convex shape constraint. In the present work, we have used a simple way to calculate the region force as our focus is to show the techniques related to the convex shape prior. It is definitely worth to investigate to use more robust and accurate region and edge forces, and this will be part of our future experiments.

Obviously, models and algorithms need to be further developed for image segmentation with shape priors.
In the future, we will investigate multi-phases image segmentation method with convex shape prior.
We will incorporate other new region force terms with convex shape prior to improve the segmentation performance.
In addition, the application and acceleration of the proposed algorithm are other interested problems.
\section*{Acknowledgment}
The authors would like to thank Professor Roland Glowinski at Department of Mathematics,
 University Houston for his valuable suggestions on the algorithm for the proposed model.
The first author is supported by the National Natural Science Foundations of China (11401171,11471101).
\bibliographystyle{ieeetr}
\bibliography{ref}
\end{document}